\documentclass[12pt]{amsart}
\usepackage[a4paper,margin=23mm]{geometry}
\usepackage{amsfonts, amsthm, amsmath, amssymb}
\usepackage{hyperref}
\hypersetup{colorlinks=false}
\renewcommand{\leq}{\leqslant}

\renewcommand{\geq}{\geqslant}

\newcommand\NN{\mathbb{N}}
\newcommand\CC{\mathbb{C}}
\newcommand\ZZ{\mathbb{Z}}
\newcommand\RR{\mathbb{R}}
\newcommand\QQ{\mathbb{Q}}

\DeclareMathOperator{\Fin}{Fin}

\usepackage{comment}
\usepackage{graphics}
\usepackage{aliascnt}

\usepackage{enumerate}
\usepackage{amsmath}
\usepackage{amsfonts}
\usepackage{amssymb}
\usepackage{amsthm}
\usepackage{comment}
\usepackage{mathtools}
\usepackage{xcolor}

\newtheorem{theorem}{Theorem}[section]
\newtheorem{corollary}[theorem]{Corollary}

\newtheorem{lemma}[theorem]{Lemma}
\newtheorem{proposition}[theorem]{Proposition}
\theoremstyle{definition}
\newtheorem{definition}[theorem]{Definition}
\newtheorem{remark}[theorem]{Remark}

\numberwithin{equation}{section}

\usepackage[backend=bibtex,
style=numeric,
isbn=false,
doi=false
]{biblatex}
\title{Gowers norms for linearly recurrent numeration systems}
\author{Pascal Jelinek}
\email{pascal.jelinek@unileoben.ac.at}
\address{
	Department Mathematics and Information Technology,
	Leoben University of Technology,
	Franz-Josef-Strasse 18, 8700 Leoben, Austria.\
}

\begin{document}
\begin{abstract}
	Gowers norms have been a key component in the proofs of many breakthrough results in connection to the sum of digits function. Spiegelhofer has used them to show that the Thue-Morse sequence has level of distribution 1 and also that it is equidistributed along cubes. Recently Gowers norms have been used to study the sum of digits function of the Zeckendorf expansion of primes. In this paper we unify the treatments of Gowers norms and give Gowers norms type estimates for a large class of linearly recurrent numeration systems.
\end{abstract}
\maketitle
\renewcommand{\thefootnote}{\fnsymbol{footnote}}
% 11A63 Radix representation; digital problems
% 11B30 Arithmetic combinatorics; higher degree uniformity
% 11B85 Automata sequences
% 11L07 Estimates on exponential sums
\footnotetext{ The author was supported by the Austrian Science Fund (FWF), Project P36137-N.}
\renewcommand{\thefootnote}{\arabic{footnote}}
\tableofcontents
\section{Introduction}
Let $G=(G_i)_{i\in\NN}$ be a strictly increasing sequence of positive integers with $G_0=1$. We can define a numeration system based on the following greedy algorithm	: Let $n$ be a natural number, suppose $G_i\leq n < G_{i+1}$. Let $\delta_i(n)=\lfloor n/G_i \rfloor$. We notice that $n-\delta_{i}(n)G_i<G_i$ and hence we can choose $\delta_{i-1}(n), \dots, \delta_{0}(n)$ similarly and get that
\[
	n = \delta_{0}(n)G_0 + \dots + \delta_{i}(n)G_i.
\]
We also note that the assumptions that $G$ is strictly increasing and that $G_0=1$ ensure that this expansion exists and is unique for any positive integer $n$. We define the sum of digits function associated to $G$ to be
\[
	s_G(n)=\sum_{i=0}\delta_i(n).
\]
We further define the truncated sum of digits function associated to $G$ to be
\begin{equation}
	s_{G,\lambda} = \sum_{i=0}^{\lambda}\delta_{i}(n)
\end{equation}
for any integer $\lambda>0$.
From now on we want to focus on a special class of such sequences $G$, namely linearly recurrent sequences. We will use a slight generalization of the definition that is used in \cite{LambergerThuswaldner2003} and \cite{MadritschThuswaldner2022}.
\begin{definition}[Linearly recurrent numeration system]
	We will say that a strictly increasing sequence $(G_i)_{i\in\NN}$ defines a linearly recurrent numeration system, if there exist $a_1,\dots,a_d\in \NN$, with $a_d>0$ such that
	\begin{itemize}
		\item $G_0=1$ and $1<G_1<\dots<G_{d-1}$ are positive integers,
		\item $G_{n}=a_1G_{n-1}+\dots+a_dG_{n-d}$ holds for each $n\geq d$,
		\item $(a_{\ell},\dots, a_d) \leq (a_1,\dots,a_{d-\ell+1})$ for each $1<\ell\leq d$, with regard to the lexicographic order.
	\end{itemize}
\end{definition}
The characteristic polynomial associated to this numeration system is $X^d-a_1X^{d-1}+\dots+a_{d-1}X+a_d$ and it has a unique dominant real root $\beta$ \cite{DrmotaGajdosik1998a}. If all conjugates of $\beta$ have absolute value $< 1$, then we say $\beta$ is Pisot.

We remark that the third condition in the definition above is also know as the Parry condition, first introduced by Parry in \cite{Parry1960}. He showed that in the related $\beta$-expansion, where $\beta$ is the dominant root from above, the strings of length $d$ that can occur are exactly those that are lexicographically below $(a_1,\dots,a_d)$. Steiner \cite{Steiner2000} showed that this condition also fully determines all possible strings for the representation of integers in this linearly recurrent base, with starting conditions $G_0=1$ and $G_i=a_{i-1}G_{i-1}+\dots + a_0G_0$ for $i<d$.

\subsection{Related works and statements of results}
The study of different numeration systems and their properties has been an area of active research for many decades. The simplest setting consists of studying the representations of positive integers in base $b$. In 1968, Gelfond \cite{Gelfond1968} proved that the sum of digits function for the base $q$ expansion is equidistributed modulo $m$. He also stated the following three problems, which have been the starting point of many developments in this area.\\
Show that each of the following sets is equidistributed modulo $m$:
\begin{enumerate}
	\item $(s_{q_1}(n),\dots,s_{q_k}(n))$, for coprime bases $q_1,\dots,q_k$
	\item $s_q(p)$, where $p$ is a prime
	\item $s_q(P(n))$, where $P$ is a polynomial supported on the natural numbers
\end{enumerate}
Only the first two problems have been solved, the first one by Kim \cite{Kim1999} and the second one by Maduit and Rivat \cite{MauduitRivat2010}. Maduit and Rivat \cite{MauduitRivat2009} further proved the third statement for quadratic polynomials, using a Gowers-3-norm. This was the first time a higher correlation measure was used in a digital problem. Following the work of Konieczny \cite{Konieczny2019}, who showed that all Gowers norms of the sum of digits function in base 2 modulo 2 are small, Spiegelhofer managed to use this result to show many breakthrough results, such as proving that the Thue--Morse sequence (which is equivalent to the sum of digits function in base 2 modulo 2) has level of distribution 1 \cite{Spiegelhofer2020}, and is equidistributed along cubes \cite{Spiegelhofer2023}. Recently in 2025, Toumi \cite{Toumi2025} generalized the level of distribution and the Gowers norm results to the sum of digits function in any integer base $b>1$ modulo an arbitrary integer $m$. 

Regarding progress in the setting of linearly recurrent numeration systems in the classical sense, many problems seem to be much harder.
An asymptotic expression for the function $s_G$ was found by Peth\H{o} and Tichy \cite{PethoTichy1989}, while Grabner and Tichy \cite{GrabnerTichy1990} proved that $\theta s_G(n)$ is equidistributed modulo 1 for each $\theta \in \RR \backslash \QQ$ using analytic methods. Drmota and Gajdosik \cite{DrmotaGajdosik1998a} showed a local limit law regarding the function $s_G$. In 2022, Madritsch and Thuswaldner \cite{MadritschThuswaldner2022} showed that each function $s_G$ corresponding to a linearly recurrent numeration system has level of distribution strictly greater than 1/2, and the level of distribution tends to 1 as $a_1$ tends to infinity.

In the special case of the Zeckendorf expansion, which was introduced by Zeckendorf \cite{Zeckendorf1972} as one of the first examples of a numeration system that is not based on powers of integers, some deeper results have been achieved: In 2018, Drmota et al \cite{DrmotaMullnerSpiegelhofer2018} showed that for the Zeckendorf sum of digits function, the Sarnak conjecture \cite{Sarnak2012}, which asks about the orthogonality of the Möbius function with respect to certain sequences, holds. Recently, in 2025, they proved that $s_F$ is equidistributed modulo m along prime numbers \cite{DrmotaMullnerSpiegelhofer2025}, hence proving the analogue of the second Gelfond problem for the Zeckendorf sum of digits function. In the latter paper they also showed that $s_F$ has small Gowers norm and that it has level of distribution 1. 

For many other classes of functions there are results about the size of their Gowers norm. In 2012, Green, Tao and Ziegler \cite{GreenTaoZiegler2012} showed that if the Gowers norm of a complex valued 1-bounded function $f:\{0,1,\dots,N\}\to\CC$ is large then the sequence correlates with a nilsequence. Similarly, in the setting of automatic sequnences, Byszewski, Konieczny and Müllner \cite{ByszewskiKoniecznyMullner2023} proved that an automatic sequence has small Gowers norm if and only if it does not correlate with a periodic function.

Our main result gives a Gowers norm like estimate for a wide class of linearly recurrent numeration systems, namely those which satisfy the so called finiteness property (Property F), introduced by Frougny and Solomak in \cite{FrougnySolomyak1992}.
\begin{definition}\label{Thm_F}
	We say that a $\beta$-expansion satisfies Property F, if for all $x,y\in \Fin_N(\beta)$ there exists an $r$ such that $x\pm y\in \Fin_{N+r}(\beta)$. We further say that a linearly recurrent numeration system satisfies Property F if the corresponding $\beta$-expansion satisfies Property F, where $\beta$ is the dominant root of the system.
\end{definition}
Here the set $\Fin_N$ is defined as
\begin{equation}
	\Fin_N:=\{x\in \QQ^+(\beta): x=\sum_{i\geq -N}\delta_i(x)\beta^i \text{ for some $\delta_i(x)\in\NN$}\}.
\end{equation}

\begin{remark}
	Frougny and Solomak showed in \cite{FrougnySolomyak1992} that a necessary but not sufficient condition for a linearly recurrent numeration system to have Property F is that it is Pisot. Some sufficient conditions on the coefficients $a_i$ such that the numeration system has Property F are known, for example that $(a_i)$ are in descending order.
\end{remark}

Now we can state the main result of this paper.

\begin{theorem}\label{Thm1}
	Let $(G_i)_{i\in \NN}$ be a linearly recurrent numeration system satisfying Property F. Then, as $\lambda \to \infty$, we have that
	\begin{equation}\label{eqGow}
		\frac 1{G_{\lambda}^{s+1}} \sum_{n_0, n_1,\dots,n_s< G_{\lambda}}\prod_{\epsilon_1,\dots,\epsilon_k \in \{0,1\}} \mathcal{C}^{|\mathbf{\epsilon}|}\; e\left(\theta s_{G,\lambda} \left(n_0 + \sum_{i=1}^{s}n_i\epsilon_i\right)\right) \ll G_{\lambda}^{-c||\theta||^2}
	\end{equation}
	for some positive constant $c$, which may depend on $s$. Here we use the usual notation that $e(x)=\exp (2\pi i x)$.
\end{theorem}
This result can also be interpreted in the setting of dynamical systems, since linearly recurrent numeration systems modulo $m$ are special morphic sequences. Hence it is a step towards the generalization of the characterization of automatic sequences regarding their Gowers norms by Byszewski, Konieczny and Müllner in 2023 \cite{ByszewskiKoniecznyMullner2023} to the setting of morphic sequences.

\begin{remark}
	We want to remark that in most cases the left-hand side of equation \eqref{eqGow} in the theorem is not a Gowers norm. This is the case only when $s_G$ is the sum of digits function for the base $b$. However, we argue that this expression is still natural to consider, as it most closely resembles the type of estimates which are needed in the results regarding the sum of digits function in base $b$ and in the Zeckendorf case. Additionally, with only minor changes one can turn the left-hand side into a proper Gowers norm. These changes do not change the method of the proof but would be unnecessarily heavy in notation for this paper, hence we omit them for clarity. However, we outline the idea of the changes here.
	
	The main change is that one cannot work directly with the function $s_G(n,\lambda)$, but one needs to use a function $g_\lambda(x)$, to be defined on the associated Rauzy fractal. The number $\lambda$ indicates the depth of the iteration we are considering of this fractal. One can show that this Rauzy fractal of depth $\lambda$ provides an analytic way of detecting the last $\lambda$ digits of a given integer in these linearly recurrent bases. Additionally, it has been shown that the Rauzy fractal provides a tiling of the space by translation, hence we can consider the function $g_\lambda(x)$ to be defined on a higher-dimensional torus. Therefore, we have an abelian group given by point-wise addition on this torus, which is needed in the definition of the Gowers norm. Replacing each sum over $n_i$ by an integral over the torus and removing the normalising weight $\frac 1 {G_{\lambda}}$ after carefully choosing the correct measure, we have now replaced the left-hand side of equation \eqref{eqGow} with a Gowers norm.
	
	Lastly, we want to remark that Property F also has a geometric interpretation in this setting. It corresponds to the condition that 0 is an inner point.
\end{remark}
 
 As an easy corollary to Theorem \ref{Thm1} we obtain an improvement of a result by Toumi \cite{Toumi2025} concerning the base $b$ sum of digits function modulo $m$.
 \begin{corollary}\label{Cor1}
 	Let $b>1$ be an integer. Then we have that
	\begin{equation}
		\frac 1{(b^{k})^s} \sum_{n_0, n_1,\dots,n_s< b^{k}}\prod_{\epsilon_1,\dots,\epsilon_k \in \{0,1\}} \mathcal{C}^{|\mathbf{\epsilon}|}\; e\left(\theta s_{b,k} \left(n_0 + \sum_{i=1}^{s}n_i\epsilon_i\right)\right) \ll (b^{k})^{-c||\theta||^2}
	\end{equation}
 	for some positive constant $c$ as $k\to \infty$.
 \end{corollary}
We emphasize here again that in this setting the left-hand side is indeed a Gowers norm corresponding to the abelian group $\ZZ/b^{k}\ZZ$.

\section{Prerequisites}
In this section we prove a few small results that will be useful later.

First we want to show that the finiteness property of the $\beta$ expansion corresponds to limited carries to the right in the linearly recurrent numeration system. For this we introduce a function $v(n,\lambda)$, which associates to each possible combination of the last $\lambda$ digits a unique number between $0$ and $G_{\lambda}$. Hence one can think of it as being the analogue of viewing an integer modulo $p^{\lambda}$ in the case of the base $p$ expansion. We define $v(n,\lambda)$ the following way:
\begin{equation}
	v(n,\lambda):=\sum_{i=0}^{\lambda}\delta_i(n)G_i\;.
\end{equation}

Now we get the following lemma.

\begin{lemma}\label{Cor_F}
	Let $(G_i)_{i\in\NN}$ define a linearly recurrent numeration system such that the associated $\beta$-expansion satisfies Property F. Let $x,y$ be such that $v(x,m)=v(y,m)=0$ for $m>d+r$. Then we have that $v(x\pm y,m-r)=0$.
\end{lemma}
\begin{proof}
	Let $x=\sum_{n\geq m} \delta_n(x)G_n$ and $y=\sum_{n\geq m} \delta_n(y)G_n$. Now consider their analogues in the $\beta$-expansions $\tilde x = \sum_{n\geq m} \delta_n(x)\beta^n$ and $\tilde{y} = \sum_{n\geq m} \delta_n(y)\beta^n$. By Property F, we have that the digits of their sum/difference are 0 for the lower digits, namely ${\delta_n}(\tilde{x}\pm \tilde{y})=0$ for $n\leq m-r$. Since the digits of $\tilde{x}\pm \tilde{y}$ in the $\beta$-expansion can be found just by using the recurrence relation, and $(G_i)_{i\in\NN}$ has the same recurrence relation for the digits in question, we have that $v(x\pm y,m-r)=0$.
\end{proof}

Now we can introduce the key lemma, which shows that if two digits are separated by enough zeros, then those digits cannot influence each other.

\begin{lemma}\label{carries}
	Let $x_0,\dots,x_s\in \NN$ be such that there exist some $k\in\NN$ and integers $m_i<G_k$ with $v(x_i,k+2r)=m_i$, for some $r$ large enough. Then we have for all $\ell\in \NN$:
	\begin{equation}
		\delta_{\ell}(x_0+\dots+x_s)=\delta_{\ell}(x_0+\dots+x_s-m_0-\dots-m_s)+\delta_{\ell}(m_0+\dots+m_s)
	\end{equation}
\end{lemma}

\begin{proof}
	The difficulty lies in the fact that in a linearly recurrent numeration system, carries both to the left and to the right can occur. Hence our aim is to show that we can bound the number of  carries to the left of $m_i$ by $r-1$, and the carries to the right of $x_i-m_i$ by $r-1$.
	
	First we observe that $m_0+\dots+m_s<(s+1)G_k<G_{k+r-1}$, for $r$ large enough. Hence we see that we have limited the left carries of $m_i$ by $r-1$.	The right carries are also bounded by $r$ by Lemma  $\ref{Cor_F}$, hence this concludes the proof.
\end{proof}

\section{Proof of the main theorem}
We will now prove the main theorem by induction on $\lambda$, the number of digits. Before we can start with the induction we need to rewrite the left-hand side of equation \eqref{eqGow}.
\begin{equation}
	\frac 1{G_{\lambda}^{s+1}} \sum_{n_0, \dots,n_s< G_{\lambda}}\prod_{\epsilon_1,\dots,\epsilon_k \in \{0,1\}} \mathcal{C}^{|\mathbf{\epsilon}|}\; e\left(\theta s_{G,\lambda} \left(n_0 + \sum_{i=1}^{s}n_i\epsilon_i\right)\right) = \sum_{n_0,\dots,n_s<G_{\mu}}S_{\mu}(n_0,\dots,n_s),
\end{equation}
for some $\mu$, where
\begin{equation}
	S_{\mu}(n_0,\dots,n_s) := \frac{1}{G_{\lambda}^{s+1}}\sum_{\substack{x_0,\dots,x_s< G_{\lambda}\\v(x_i,\mu)=v(n_i,\mu)\, \forall 0\leq i\leq s}}\prod_{\epsilon_1,\dots,\epsilon_s \in \{0,1\}} \mathcal{C}^{|\mathbf{\epsilon}|}\; e\left(\theta s_{G,\lambda} \left(n + \sum_{i=1}^{s}y_i\epsilon_i\right)\right)\; .
\end{equation}
Furthermore, we see that for any $\mu'>\mu$ we have that
\begin{equation}
	S_{\mu}(n_0,\dots,n_s) = \sum_{\substack{\mathbf{n'}<G_{\mu'}\\v(\mathbf{n'},\mu)=\mathbf{n}}}S_{\mu'}(n_0',\dots,n_s')\;.
\end{equation}
We are interested in the number of summands of the above identity. Let $S:=\{(\delta_{\mu-d+1},\dots,\delta_{\mu}):(\delta_{\mu-d+1},\dots,\delta_{\mu})\leq(a_{1},\dots,a_{d})\}$. Let $\delta_i'\in S$ and $\boldsymbol{\delta'}=(\delta_1',\dots,\delta_d')$. Moreover let
\begin{align}
	M(\mu,\mu',\mathbf{n},\boldsymbol{\delta'}) &:= \{0\leq \mathbf{n'}<G_{\mu'}^{s+1}: (\delta_{\mu'}(\mathbf{n'}),\dots,\delta_{\mu'-d}(\mathbf{n'}))=\boldsymbol{\delta'}, v(\mathbf{n'},\mu)=\mathbf{n}\}\\
	\label{sizeN}N(\mu,\mu',\mathbf{n},\boldsymbol{\delta'}) &:= |M(\mu,\mu',\mathbf{n},\boldsymbol{\delta'})|<\prod_{i=0}^{s}G_{\mu'-\mu}<c^{s+1}\beta^{(\mu'-\mu)(s+1)},
\end{align}
where $\beta$ is the dominant root of the recurrence and $c$ is some positive constant.

We have now introduced the relevant notation for the induction. For the sake of keeping the main proof concise, we first show a proposition which will be the main ingredient in the induction step later on.

\begin{proposition}\label{Prop_savings}
	For $\mu\leq \lambda-3r-(s+2)d-1$, let $\mu'=\mu+3r+(s+2)d+1$. Then we have that
	\begin{equation}
		|S_{\mu}(\mathbf{n})| \leq \left(1-\frac{4||h\theta||^2}{c^{s+1}\beta^{(\mu'-\mu)(s+1)}}\right) \sum_{\boldsymbol{\delta'}\in S^{s+1} }N(\mu,\mu',\mathbf{n},\boldsymbol{\delta'}) \times \max_{\mathbf{n'}\in M(\mu,\mu',\mathbf{n},\boldsymbol{\delta'})}|S_{\mu'}(\mathbf{n'})|\;,
	\end{equation}
	where $h=a_0-1+a_1+\dots+a_d$.
\end{proposition}

\begin{proof}
	Recall that
	\begin{equation}
	|S_{\mu}(\mathbf{n})|= \left|\sum_{\substack{\mathbf{n'}<G_{\mu'}\\ v(\mathbf{n'},\mu)=\mathbf{n}}}S_{\mu'}(\mathbf{n'})\right|\;.
	\end{equation}
	Grouping the summands in terms of the $M$ and using the triangle inequality we get that
	\begin{equation}
		|S_{\mu}(\mathbf{n})|= \sum_{\boldsymbol{\delta'}\in S^{s+1} }\left|\sum_{\mathbf{n'}\in M(\mu,\mu',\mathbf{n}, \boldsymbol{\delta'})}S_{\mu'}(\mathbf{n'})\right|\;.
	\end{equation}
	Hence it suffices to construct for each $\boldsymbol{\delta'}$ integers $n_i^{(1)},n_i^{(2)}\in M(\mu,\mu',\mathbf{n},\boldsymbol{\delta'})$ which cancel at least partially. Let each number consist of 3 sections. The leftmost digits are given by the set $M$, the rightmost digits are the digits $n_s$ which we cannot influence and the digits in the middle section are chosen freely. Furthermore, we separate each of these sections with enough zeros, such that there are no carries when adding up $s+1$ of these numbers. We now construct these numbers explicitly:
	\begin{align*}
		n_s^{(1)} = n_s^{(2)} =& \sum_{j=0}^{d}(\delta_i')_{d-j}{d}G_{\mu'-1-j}+(a_1-1)G_{\mu+2r+d}+a_2G_{\mu+2r+d-1} + \dots+a_dG_{\mu+2r+1}+n_s \\
		n_i^{(1)} = n_i^{(2)} =& \sum_{j=0}^{d}(\delta_i')_{d-j}{d}G_{\mu'-1-j}+(a_1-1)G_{\mu+2r+d(s+1-i)}+a_2G_{\mu+2r+d-1} +\dots + \\
		&(a_d-1)G_{\mu+2r+1+d(s-i)} + G_{\mu+2r+d(s-i)}+n_i \text{ for } 1\leq i\leq s-1\\
		n_0^{(1)} =& \sum_{j=0}^{d}(\delta_i')_{d-j}{d}G_{\mu'-1-j}+(a_1)G_{\mu+2r+d(s+1)}+a_2G_{\mu+2r+d-1} + a_2G_{\mu+2r+d-1} +\dots + \\
		&(a_d-1)G_{\mu+2r+1+d(s)} + G_{\mu+2r+d(s)}+n_0 \\
		n_0^{(2)} =& \sum_{j=0}^{d}(\delta_i')_{d-j}{d}G_{\mu'-1-j}+ G_{\mu+2r+d(s+1)+1} + (a_1-1)G_{mu+2r+d(s+1)}+\dots + \\ &(a_d-1)G_{\mu+2r+1+d(s)} + G_{\mu+2r+d(s)}+ n_0 \\
	\end{align*}
	Following similar arguments as in the proof of Proposition 2.1 in Steiner \cite{Steiner2000}, we infer that the assumption $(a_{\ell},\dots,a_d)\leq (a_1,\dots,a_{d-\ell+1})$ implies that all of the above integers are written in the canonical base $G$ expansion.
	
	Using Lemma \ref{carries}, we see that by construction there are no carries between any of the three sections, hence we can treat each of them separately. Also note that the first and the last section of $\mathbf{n^{(1)}},\mathbf{n^{(2)}}$ are equal, hence their sum of digits is equal when restricting to those sections. Therefore we can evaluate their contribution trivially and it suffices to study the differences that appear in the values of the sum of digits function for the middle sum.
	
	By construction, we have that
	\begin{equation}
		 s_G\left(n_0^{(1)}+\sum_{i=1}^s\epsilon_in_0^{(1)},\mu'\right)= s_G\left(n_0^{(2)}+\sum_{i=1}^s\epsilon_in_0^{(2)},\mu'\right)
	\end{equation}
	if and only if not all of the $\epsilon_i$ are $1$.\\
	In the latter case we see that in each of the middle blocks the difference between the values of the sum of digits function is equal to $a_0-1+a_1+\dots+a_i$ for some $i$. Since all blocks are independent of each other, we see that by construction
	\begin{equation}
		s_G\left(n_0^{(1)}+\sum_{i=1}^s\epsilon_in_0^{(1)},\mu'\right)= s_G\left(n_0^{(2)}+\sum_{i=1}^s\epsilon_in_0^{(2)},\mu'\right)+h\; .
	\end{equation}
	For $n^{(1)}$, the expression evaluates to $ G_{\mu+2r+d(s+1)+1}$, whereas for $n^{(2)}$, we are left with $G_{\mu+2r+d(s+1)+1} + (a_1-1)G_{mu+2r+d(s+1)}+\dots+(a_d)G_{\mu+2r+1+d(s)}$ in the middle.
	Hence combining all calculations yields:
	\begin{equation}
	S_{\mu'}(\mathbf{n^{(1)}})=S_{\mu'}(\mathbf{n^{(2)}})e((-1)^sh\theta)\;.
	\end{equation}
	Using the estimate that $|1+e((-1)^sh\theta)|\leq 2|cos(\pi ||h\theta||)| \leq 2-4||h\theta||^2$, we get
	\begin{equation}
	\left|\sum_{\mathbf{n'}\in M(\mu,\mu',\mathbf{n}, \boldsymbol{\delta'})}S_{\mu'}(\mathbf{n'})\right| \leq (N(\mu,\mu',\mathbf{n})-4||h\theta||^2)\max_{\mathbf{n'}\in M(\mu,\mu',\mathbf{n})}|S_{\mu'}(\mathbf{n'})|\;.
\end{equation}
	Therefore, using that $N(\mu,\mu',\mathbf{n},\boldsymbol{\delta'})<c^{s+1}\beta^{(\mu'-\mu)(s+1)}$, the statement follows.	
\end{proof}

Now we can finally prove Theorem \ref{Thm1}. 

\begin{proof}[Proof of Theorem \ref{Thm1}]
	First note that for fixed $\mu,\mu'$ and $\boldsymbol{\delta'}$, the quantity $N(\mu,\mu',\mathbf{n},\boldsymbol{\delta'})$ depends only on the value of $\delta_{\mu}(\mathbf{n},\dots,\delta_{\mu-d}(\mathbf{n}))=\boldsymbol{\delta}$, hence we use the notation $N(\mu,\mu',\boldsymbol{\delta},\boldsymbol{\delta'})$ interchangeably with $N(\mu,\mu',\mathbf{n},\boldsymbol{\delta'})$.

	Now we write $\lambda=(3b+3)r\lambda' + k$, where $k\leq (3b+3)r-1$ and we want to prove the following statement by induction on $\ell$:
	\begin{align*}
		\frac 1{G_{\lambda}^{s+1}} &\sum_{n_0, n_1,\dots,n_s< G_{\lambda}}\prod_{\epsilon_1,\dots,\epsilon_k \in \{0,1\}} \mathcal{C}^{|\mathbf{\epsilon}|}\; e\left(\theta s_{G,\lambda} \left(n_0 + \sum_{i=1}^{s}n_i\epsilon_i\right)\right)\\
		\leq &\left(1-\frac{4||h\theta||^2}{c^{s+1}\beta^{((3b+3)r)(s+1)}}\right)^{\ell} \sum_{n_0,\dots,n_s<G_{k}}\sum_{\boldsymbol{\delta'}\in S^{s+1} }N(k,k+\ell((3b+2 )r+4),\mathbf{n},\boldsymbol{\delta'})\\
		&\times \max_{\mathbf{n'}\in M(k,k+\ell(3b+3)r,\mathbf{n},\boldsymbol{\delta'})} |S_{k+\ell(3b+3)r}(\mathbf{n'})|
	\end{align*}
	
	The induction base $\ell = 0$ clearly holds, since
	\begin{equation}
		N(k,k,\mathbf{n},\boldsymbol{\delta'})= \begin{cases}
			1,& \text{if } (\delta_{k}(\mathbf{n}),\dots,\delta_{k-d}(\mathbf{n}))=\boldsymbol{\delta'}\\
			0,& \text{otherwise}.
		\end{cases}
	\end{equation}
	Now suppose that the statement holds for $\ell<\lambda'$. Applying Proposition \ref{Prop_savings}, we find that
	\begin{align*}
		&|S_{k+\ell((3b+3)r)}(\mathbf{n'})|\\
		&\leq \left(1-\frac{||h\theta||^2}{c^{s+1}\beta^{((3b+3)r)(s+1)}}\right) \sum_{\boldsymbol{\delta''}\in S^{s+1}} N(k+\ell(3b+3)r,k+(\ell+1)((3b+3)r),\mathbf{n'},\boldsymbol{\delta''})\\ 
		&\hspace{0.3cm}\times \max_{\mathbf{n''}\in M(k+\ell((3b+3)r),k+(\ell+1)(3b+3)r,\mathbf{n'},\boldsymbol{\delta''})}
		|S_{k+(\ell+1)(3b+3)r}(\mathbf{n''})|\\
		&\leq \left(1-\frac{||h\theta||^2}{c^{s+1}\beta^{((3b+3)r)(s+1)}}\right) \sum_{\boldsymbol{\delta''}\in S^{s+1}} N(k+\ell(3b+3)r,k+(\ell+1)((3b+3)r),\mathbf{n'},\boldsymbol{\delta''})\\
		&\hspace{0.3cm}\times \max_{\mathbf{n''}\in M(k,k+(\ell+1)((3b+3)r),\mathbf{n},\boldsymbol{\delta''})} 
		|S_{k+(\ell+1)(3b+3)r}(\mathbf{n''})|\;.
	\end{align*}
	Combining this with the following identity finishes the induction step:
	\begin{align*}
		\sum_{\boldsymbol{\delta'}\in S^{s+1}} &N(k,k+\ell((3b+3)r),\mathbf{n},\boldsymbol{\delta'})\;N(k+\ell((3b+3)r),k+(\ell+1)((3b+3)r,\boldsymbol{\delta'},\boldsymbol{\delta''}))\\
		&=N(k,k+(\ell+1)((3b+3)r),\mathbf{n},\boldsymbol{\delta''})
	\end{align*}
	
	To finish the proof of the theorem, we estimate $|S_\lambda(\mathbf{n'})|$ trivially:
	\begin{equation}
		|S_\lambda(\mathbf{n'})| \leq \frac{1}{G_{\lambda}^{s+1}}.
	\end{equation}
	Combining this with equation \eqref{sizeN}, we have
	\begin{equation}
		|S_\lambda(\mathbf{n'})|N(k,\lambda,\mathbf{n},\boldsymbol{\delta''}) \leq \prod_{i=0}^s \frac{c^{s+1}\beta^{((3b+3)r)}}{G_{\lambda}}\;,
	\end{equation}
	which we can bound by an absolute constant (which might depend on $s$). Therefore, we have that
	\begin{align*}
		\frac 1{G_{\lambda}^{s+1}} &\sum_{n_0, n_1,\dots,n_s< G_{\lambda}}\prod_{\epsilon_1,\dots,\epsilon_k \in \{0,1\}} \mathcal{C}^{|\mathbf{\epsilon}|}\; e\left(\theta s_{G,\lambda} \left(n_0 + \sum_{i=1}^{s}n_i\epsilon_i\right)\right)\\
		 &\leq
		\left(1-\frac{||h\theta||^2}{c^{s+1}\beta^{((3b+3)r)(s+1)}}\right)^{\lambda'} \sum_{n_0,\dots,n_s<G_{k}}\sum_{\boldsymbol{\delta'}\in S^{s+1}}N(k,\lambda,\mathbf{n},\boldsymbol{\delta'})
		\times \max |S_{\lambda}(\mathbf{n'})|\\
		&\ll_s \left(1-\frac{||h\theta||^2}{c^{s+1}\beta^{((3b+3)r)(s+1)}}\right)^{\lambda'}\\
		&\ll \exp(-\tilde{c}(s)\lambda||h\theta||^2)\;,
	\end{align*}
	which concludes the proof.
\end{proof}
\printbibliography
\end{document}